\pdfoutput=1
\documentclass[10pt]{extarticle}

\usepackage[left=0.8in,right=0.8in,top=0.8in,bottom=1.2in]{geometry}
\usepackage{amsmath, amssymb, amsthm, latexsym, amsfonts, indentfirst, xcolor, mathtools, manfnt, microtype, subfig, cite, graphicx}
\usepackage[utf8]{inputenc} 
\usepackage[english]{babel}

\usepackage[breaklinks]{hyperref}
\usepackage{cleveref}
\hypersetup{
	colorlinks = true, 
	urlcolor = cyan, 
	linkcolor = teal, 
	citecolor = cyan 
}

\let\OLDthebibliography\thebibliography
\renewcommand\thebibliography[1]{
	\OLDthebibliography{#1}
	\setlength{\parskip}{0pt}
	\setlength{\itemsep}{0pt plus 0.3ex}
}

\newtheorem{Theorem}{Theorem}

\newtheorem{Lemma}{Lemma}

\newtheorem{Proposition}{Proposition}

\newcommand{\sm}{\!\setminus\!}
\newcommand{\R}{\mathbb R}

\newcommand{\N}{\mathbb N}

\newcommand{\x}{\mathbf x}
\newcommand{\y}{\mathbf y}
\newcommand{\z}{\mathbf z}
\newcommand{\bbb}{\mathbf b}
\newcommand{\rr}{\mathbf r}
\newcommand{\M}{\mathcal M}
\newcommand{\pot}{2}
\newcommand{\G}{\mathcal G}
\newcommand{\V}{\mathbf v}
\newcommand{\W}{\mathbf w}
\newcommand{\U}{\mathbf u}

\begin{document}

\title{Monochromatic infinite sets in Minkowski planes}
\author{
	N\'ora Frankl\thanks{The Open University, Milton Keynes, UK, and Alfr\'ed R\'enyi Institute of Mathematics, Budapest, Hungary. \newline Email:~\href{mailto:nora.frankl@open.ac.uk}{\tt nora.frankl@open.ac.uk}.}
	\and
	Panna Geh\'er\thanks{E\"otv\"os Lor\'and University, Budapest, Hungary,  and Alfr\'ed R\'enyi Institute of Mathematics, Budapest, Hungary. \newline Email:~\href{mailto:geher.panna@ttk.elte.hu}{\tt geher.panna@ttk.elte.hu}.}
	\and
	Arsenii Sagdeev\thanks{KIT, Karlsruhe, Germany and Alfr\'ed R\'enyi Institute of Mathematics, Budapest, Hungary. Email:~\href{mailto:sagdeevarsenii@gmail.com}{\tt sagdeevarsenii@gmail.com}.}
	\and
	G\'eza T\'oth\thanks{Alfr\'ed R\'enyi Institute of Mathematics, Budapest, Hungary. Email:~\href{mailto:geza@renyi.hu}{\tt geza@renyi.hu}.}
}

\maketitle

\begin{abstract}
	We prove that for any  $\ell_p$-norm in the plane with $1< p< \infty$ and for every infinite $\M \subset \R^2$, there exists a two-colouring of the plane such that no isometric copy of $\M$ is monochromatic. On the contrary, we show that for every polygonal norm (that is, the unit ball is a polygon) in the plane, there exists an infinite $\M \subset \R^2$ such that for every two-colouring of the plane there exists a monochromatic isometric copy of $\M$.
\end{abstract}



\section{Introduction} \label{sec1}

The following question of Nelson from 1950 has greatly influenced modern combinatorial geometry:
what is the minimum number of colours $\chi(\R^2)$ needed to colour $\R^2$ such that no two points unit Euclidean distance apart are of the same colour? Despite all the attention it received, this seventy-year-old problem is still open. The state of the art is limited by the inequalities $5 \le \chi(\R^2) \le 7$. The lower bound is relatively recent, obtained first by de Grey~\cite{degrey} and shortly afterwards by Exoo and Ismailescu \cite{Exoo} with a different proof. For more about the problem, see a survey in Soifer's book~\cite{Soifer}.
We consider generalisations of the original problem, where we impose different restrictions on the colouring.

The $n$-dimensional space $\R^n$, together with a norm 
$N$ is called a {\em Minkowski space} (or Minkowski plane, for $n=2$) \cite{MSW01}. 
The \textit{$N$-distance} between $\x, \, \y \in \R^n$ is denoted by $\|\x-\y\|_N$.  For $r>0$, the set $\{\y \in \R^n: \|\y-\x\|_N\le r\}$ is called the \textit{$N$-ball of radius $r$ centered at $\x$} and is denoted by $B_N(\x,r)$. The norm $N$ is called \textit{polygonal} if its unit ball is a polytope. It is  called {\it strictly convex} if the unit $N$-ball is strictly convex. In other words, $N$ is strictly convex if for each $\x, \, \y \in \R^n$ such that $\|\x\|_N=\|\y\|_N=1$ and for all $0 < \lambda < 1$, we have $\|\lambda \x + (1- \lambda) \y\|_N<1$. Note that this is equivalent to the the fact that the equality $\|\x-\y\|_N+\|\y-\z\|_N=\|\x-\z\|_N$ implies that $\x$, $\y$ and $\z$ are collinear.
Given a subset $\M \subset \R^n$, we call $\M' \subset \R^n$ an {\it $N$-isometric copy of $\M$} if there exists a bijection $f: \M \to \M'$ such that $\|\x-\y\|_N = \|f(\x)-f(\y)\|_N$ for all $\x, \, \y \in \M$. We emphasise that unless $N$ is strictly convex, not all $N$-isometric copies of a collinear set are collinear.
Finally, let the \textit{chromatic number} $\chi(\R^n, N, \M)$ be the minimum number of colours needed to colour $\R^n$ such that no $N$-isometric copy of $\M$ is monochromatic. In the simplest case when the forbidden configuration $\M$ is a two-point set, we denote this chromatic number by $\chi(\R^n, N)$ for shorthand.

A systematic study of these quantities in the Euclidean space, that is, for the $\ell_2$-norm, dates back to the works of Erd\H{o}s, Graham, Montgomery, Rothschild, Spencer, and Straus~\cite{EGMRSS1, EGMRSS2, EGMRSS3}. For further history and details we refer to Graham's survey~\cite{Graham2017}. Research so far, e.g.~\cite{FrRod, KSZ, Leader} mostly focused on \emph{finite} configurations for the following reason.
For every dimension $n \in \N$ and every \textit{infinite} $\M \subset \R^n$, we have $\chi(\R^n, \ell_2, \M)=2$, see~\cite[Theorem~19]{EGMRSS2}. 
We generalise this result for all strictly convex $\ell_p$-norms in the plane. Recall that for $\x = (x_1, \, \dots, \, x_n) \in \R^n$, its \textit{$\ell_p$-norm} is given by $\|\x\|_p = (\sum_{i=1}^{n}|x_i|^p)^{1/p}$ for each $1 \le  p < \infty$, and by $\|\x\|_\infty = \max_{i}|x_i|$ in case $p=\infty$. Observe that the $\ell_1$- and $\ell_\infty$-norms are polygonal, while $\ell_p$-norms are strictly convex for all $1 < p < \infty$.

\begin{Theorem} \label{th_lp_plane}
	Let $1 < p < \infty$ and $\M \subset \R^2$ be infinite. Then we have $\chi(\R^2, \ell_p, \M)=2$. In other words, there exists a two-colouring of the plane such that no $\ell_p$-isometric copy of $\M$ is monochromatic. 
\end{Theorem}

The $\ell_\infty$-norm is distinguished from the strictly convex ones in this context for example by the following two properties
of infinite \textit{geometric progressions} $	\G(q) \coloneqq \{0, \, 1, \, 1+q, \, 1+q+q^2, \, \dots\} $
obtained in~\cite[Section~2.3]{FKS}. First, for every $n \in \N$, we have $\chi\big(\R^n, \ell_\infty, \G(q)\big) = n+1$ whenever $q$ is sufficiently small in terms of $n$. Second, for each positive $q<1/32$, we have $\chi\big(\R^n, \ell_\infty, \G(q)\big) > \log_3n$ for all $n \in \N$. Our second result extends the former property to all polygonal norms in the plane.

\begin{Theorem} \label{th_polygon}
	Let $N$ be a polygonal norm in the plane, and $\lambda$ be the smallest $N$-distance between two consecutive vertices of its unit disc. Then for each positive $q < \lambda/(1+\lambda)$, we have $\chi\big(\R^2, N, \G(q)\big) \ge 3$. In other words, for every two-colouring of the plane, there exists a monochromatic $N$-isometric copy of $\G(q)$.
\end{Theorem}

\noindent {\bf Paper outline.} In the next two sections, we prove Theorems~\ref{th_lp_plane} and~\ref{th_polygon}, respectively. In Section~\ref{sec:conc}, we discuss the limitations of our tools and state several open problems.

\section{Strictly convex \texorpdfstring{$\ell_p$}{lp}-norms -- Proof of Theorem~\ref{th_lp_plane}} \label{sec2}

To prove Theorem \ref{th_lp_plane}, we reformulate it using the terminology of hypergraphs.

\subsection{Notation and preliminaries}

As usual, the \textit{chromatic number} of a hypergraph $\mathcal{H}=(V,E)$ is defined as the minimum number of colours needed to colour its vertices such that no edge is monochromatic. A function $\varphi: V \to \N$ is called a \textit{polychromatic $\N$-colouring} of $\mathcal{H}$ if for all $e \in E$, $i \in \N$ there is $v \in e$ such that $\varphi(v)=i$. It is clear that if there exists a polychromatic $\N$-colouring $\varphi$ of a hypergraph $\mathcal{H}$, then its chromatic number is equal to $2$. For instance, one can colour each vertex $v \in V$ according to the parity of $\varphi(v)$ to prove this simple implication.

For a norm $N$ on $\R^n$ and for an infinite subset $\M \subset \R^n$, let $\mathcal{H}(\R^n, N, \M)$ be the hypergraph with all points of $\R^n$ being its vertices and all $N$-isometric copies of $\M$ being its edges.  With this notation, it is easy to see that \Cref{th_lp_plane} states that the chromatic number of $\mathcal{H}(\R^2, \ell_p, \M)$ is equal to $2$ for every $1 < p < \infty$. As we have discussed in the previous paragraph, this would follow from the next stronger statement.

\begin{Theorem} \label{th_lp_plane_strong}
	Let $1 < p < \infty$ and $\M \subset \R^2$ be infinite. Then there exists a polychromatic $\N$-colouring of the hypergraph $\mathcal{H}(\R^2, \ell_p, \M)$. That is, the points of $\R^2$ can be coloured with countably many colours so that each  $N$-isometric copy 
 of $\M$
 contains a point of each colour class.
\end{Theorem}

It seems natural to conjecture that a polychromatic $\N$-colouring of $\mathcal{H}(\R^n, N, \M)$ should exist whenever the norm is strictly convex, but this question remains open. However, in case $\M$ is unbounded, it is not hard to construct such a colouring explicitly for any, not necessarily strictly convex norm.

\begin{Proposition} \label{unboundedM}
	Let $N$ be a norm on $\R^n$ and an infinite $\M \subset \R^n$ be unbounded. Then there exists a polychromatic $\N$-colouring of the hypergraph $\mathcal{H}(\R^n, N, \M)$. 
\end{Proposition}
\begin{proof}
	We will find the desired colouring along a recursively constructed sequence of nested $N$-balls centred at the origin $\mathbf{0}\in \R^n$ with sufficiently quickly growing radii and alternating colours.
	
	Choose an element $\x^0\in \M$ and let $r_1=1$. For all $i \in \N$, on the $i$-th step, we pick an arbitrary $\x^i \in \M$ such that $\|\x^i-\x^0\|_N>2r_i$. This is possible since $\M$ is unbounded. Set $r_{i+1} = r_i+\|\x^i-\x^0\|_N$. Consider the sequence of $N$-balls $B_i = B_N(\mathbf{0},r_i)$, $i \in \N$. First, observe that $r_{i+1} > 3r_i$ for all $i \in \N$ by construction. Therefore, $\R^n$ is split into the disjoint union of `rings' $B_{i}\sm B_{i-1}$, $i \in \N$. (Here it is convenient to define $B_0=\varnothing$ for uniformity.) Second, we claim that every $N$-isometric copy $\M'$ of $\M$ intersects all but finitely many of these rings. Indeed, let $\y^0, \, \y^1, \, \dots$ be the elements of $\M'$ that correspond to $\x^0, \, \x^1, \, \dots$ from $\M$, respectively. Let $i\in \N$ be such that $\y^0 \in B_{i}$. It is sufficient to check that for all $j \ge 0$, we have $\y^{i+j} \in B_{i+j+1}\sm B_{i+j}$. To see this, we note that
	\begin{equation*}
		\|\y^{i+j}\|_N \ge \|\y^{i+j}-\y^0\|_N-\|\y^{0}\|_N = \|\x^{i+j}-\x^0\|_N-\|\y^{0}\|_N \ge 2r_{i+j}-r_i > r_{i+j}.  
	\end{equation*}
	Thus, we have $\y^{i+j} \notin B_{i+j}$. Similarly,
	\begin{equation*}
		\|\y^{i+j}\|_N \le \|\y^{i+j}-\y^0\|_N+\|\y^{0}\|_N \le \|\x^{i+j}-\x^0\|_N+r_i \le r_{i+j+1}.  
	\end{equation*}
	Hence, we have $\y^{i+j} \in B_{i+j+1}$.
	
	Now it is easy to construct the desired polychromatic $\N$-colouring by assigning the colour $\psi(i)$ to all points of the $i$-th ring $B_{i}\sm B_{i-1}$, where $\psi: \N \to \N$ is an arbitrary function that takes every natural number as its value infinitely often. (For instance, we can set $\psi\left(\frac{k(k-1)}{2}+i\right) = i$ for all $k \in \N$, $i \in [k]$. On the first ten natural numbers, this function takes the values $1, \ 1,2, \ 1,2,3, \ 1,2,3,4$, respectively.)
\end{proof}

This implies Theorem \ref{th_lp_plane_strong} when $\M$ is unbounded. 
The case of a bounded infinite $\M$ is substantially harder. The existence of a polychromatic $\N$-colouring of the hypergraph $\mathcal{H}(\R^n, \ell_2, \M)$ is guaranteed by~\cite[Theorem~19]{EGMRSS2} for a Euclidean space of arbitrary dimension. In the centre of that argument was the following theorem due to Erd\H{o}s and Hajnal.

\begin{Theorem}[Erd\H{o}s--Hajnal,
\cite{ErdHaj}] \label{ErdosHajnal}
	Let $k$ be an integer and $\mathcal{H}=(V,E)$ be a hypergraph such that every $e \in E$ is infinite and that $|e_1\cap e_2|<k$ for all distinct $e_1,e_2 \in E$. Then there exists a polychromatic $\N$-colouring of $\mathcal{H}$.
\end{Theorem}

Note that this theorem was not stated explicitly in \cite{ErdHaj}. However, it follows from \cite[Theorem~8.b]{ErdHaj}. Indeed, the latter result implies that under the conditions of \Cref{ErdosHajnal}, there exists a subset $V_1 \subset V$ such that the intersection $e \cap V_1$ is finite and non-empty for every edge $e \in E$. Observe that the `projection' hypergraph $\mathcal{H}'=(V',E')$, where $V' = V\sm V_1$ and $E' = \{e \cap V': e \in E\}$, still satisfies the conditions of \Cref{ErdosHajnal}. Thus, there exists another subset $V_2 \subset V'$ such that the intersection $e \cap V_2$ is finite and non-empty for every edge $e \in E'$. Proceeding in the same vein, we get an infinite sequence $V_1, \, V_2, \, \dots$ of disjoint subsets of $V$ such that each $V_i$ intersects every $e \in E$. Now mapping all the vertices of $V_i$ to $i$ for all $i \in \N$, we obtain the desired polychromatic $\N$-colouring of $\mathcal{H}$. In this construction, the vertices from $V\sm (V_1\cup V_2 \cup \dots)$ are `redundant', and can be mapped arbitrarily.

To extend the argument from~\cite{EGMRSS2} to other $\ell_p$-norms in the plane, we need the following strengthening of the previous theorem.

\begin{Lemma} \label{infcommon}
	Let $k$ be an integer and $\mathcal{H}=(V,E)$ be a hypergraph. For an edge $e\in E$, let $e'$ be its intersection 
    with all the edges that share at least $k$ vertices with $e$, i.e. let
	\begin{equation*}
e'\coloneqq \hspace{-2mm} \bigcap_{\substack{f \in E \\ |e\cap f|\ge k}}\hspace{-3mm}f \subset V.
	\end{equation*}
	If $e'$ is infinite for every edge $e$, then there exists a polychromatic $\N$-colouring of $\mathcal{H}$.
\end{Lemma}
\begin{proof}
 First, we observe that if for two edges $e_1$ and $e_2$ the sets $e_1'$ and $e_2'$ are distinct, then they share less than $k$ vertices. Indeed, it is clear that every $f \in E$ that shares at least $k$ vertices with $e_1$ or $e_2$, including the edges $e_1$ and $e_2$ themselves, also contains the intersection $e_1'\cap e_2'$ by construction. If the cardinality of the latter intersection is at least $k$, then we conclude that $|e_1\cap f|\ge k$ if and only if $|e_2\cap f|\ge k$ for all $f \in E$. Hence, both $e_1'$ and $e_2'$ are defined as the intersection of the same family of edges, and so they coincide even if $e_1,e_2 \in E$ are distinct.
	
	Consider the hypergraph $\mathcal{H}'=(V,E')$, where $E'=\{e':e \in E\}$. By our assumption, it satisfies the conditions of \Cref{ErdosHajnal}, and thus there exists a polychromatic $\N$-colouring $\varphi:V \to \N$ of $\mathcal{H}'$. Since $e'\subset e$ for all $e\in E$, it is easy to see that $\varphi$ is a polychromatic $\N$-colouring of $\mathcal{H}$ as well.
\end{proof}

The last ingredient in our proof of \Cref{th_lp_plane_strong} is the following structural lemma, ensuring that in every infinite $\M$, there exists an infinite subset $\M' \subset \M$ such that all the distances between its points are pairwise distinct. Even though later we will use this lemma only in case of the $\ell_p$-norm in the plane, we state the lemma in its general form, because the proof is universal.

\begin{Lemma} \label{distDist}
	Let $N$ be a norm on $\R^n$ and $\M \subset \R^n$ be an infinite bounded set. Then there exists an infinite $\M' \subset \M$ such that all the distances $\|\x-\x'\|_N$, $\x, \, \x' \in \M'$ are pairwise distinct.
\end{Lemma}
\begin{proof} Observe that the set of accumulation points of $\M$ is non-empty since $\M$ is infinite and bounded. Let $\y \in \R^n$ be an arbitrary accumulation point of $\M$. Consider a sequence $X=\{\x^1, \, \x^2, \, \dots\}$ of points of $\M$ that converges to $\y$ with at least exponential speed. More specifically, we begin with some $\x^1\in \M\sm\y$, and construct this sequence recursively by picking on the $i$-th step an arbitrary $\x^{i+1}\in \M\sm\y$ such that $\|\x^{i+1}-\y\|_N \le \frac{1}{3}\|\x^{i}-\y\|_N$. Observe that for all $i, \, j, \, k_1, \, k_2 \in \N$ such that $i<j$, $i<k_1$, $j<k_2$, we have
	\begin{align*}
		\|\x^i-\x^{k_1}\|_N &\ge \|\x^i-\y\|_N - \|\x^{k_1}-\y\|_N \ge \|\x^i-\y\|_N - \|\x^{i+1}-\y\|_N \ge \frac{2}{3}\|\x^i-\y\|_N, \\
		\|\x^j-\x^{k_2}\|_N &\le \|\x^j-\y\|_N + \|\x^{k_2}-\y\|_N \le \|\x^{i+1}-\y\|_N + \|\x^{i+2}-\y\|_N \le \frac{4}{9}\|\x^i-\y\|_N.
	\end{align*}
	Therefore, if $\|\x^i-\x^{k_1}\|_N = \|\x^j-\x^{k_2}\|_N$ for some $i<k_1$, $j<k_2$, then we must have $i=j$. In other words, if two point pairs of $X$ are at the same distance apart, then the 
 smaller indices in these pairs coincide. It is clear that subsequences of $X$ inherit this property as well.

For any $\x_i, \x_j\in X$, we say that $\x_i$ is {\em smaller} (resp. {\em larger}) than $\x_j$
if $i$ is smaller (resp. larger) than $j$.
We construct now a subsequence $Z=\{\z^1, \, \z^2, \, \dots\}\subseteq X$ recursively. 
At step $i$, we construct $\z^i$ and we might delete some of the points of $X$. Moreover, each $\z^i$ will be be coloured red or blue. Let $\z^1=\x^1$. Suppose that $i\ge 1$ and we have already constructed $\z^1, \, \ldots, \, \z^i$. 
If there is an $N$-sphere centred at $\z^i$ that contains infinitely many points of $X$, then delete all other points of $X$ and let $\z^i$ be blue. 
Otherwise, every $N$-sphere centred at $\z^i$ contains finitely many points of $X$. In this case, keep only one point of $X$ on each of these spheres, delete the others, and let $\z^i$ be red. Let $\z^{i+1}$ be the smallest point of $X$ which is larger than $\z^i$. Let $Z=\{\z^1, \, \z^2, \,\dots\}$ be the result of this recursive procedure.
 
If $\bbb^1, \, \ldots, \, \bbb^m$ is a subsequence of blue points of $Z$, then for all $1\le i<j<k\le m$, we have $\|\bbb^i-\bbb^j\|=\|\bbb^i-\bbb^k\|$. A result of Polyanskii \cite[Theorem~3]{pol17} implies that in this case $m=O(3^nn)$.

Delete all the blue points of $Z$ and denote the resulting infinite `red' subsequence
by $\M'=\{\rr^1, \, \rr^2, \, \dots\}$. Recall that if $\|\rr^i-\rr^{k_1}\|_N = \|\rr^j-\rr^{k_2}\|_N$ for some $i<k_1$, $j<k_2$, then we must have $i=j$, but since $\rr^i$ is red, we must have $k_1=k_2$ as well. That is, all the distances in $\M$ are pairwise distinct, as desired.
\end{proof}

\noindent {\bf Remark.} \Cref{distDist} holds also when $\M$ is unbounded, the argument is very similar. Moreover, the result, for any $\M$, easily 
follows from a more general theorem of Erd\H os and Rado~\cite{ER50} and 
the aforementioned result of Polyanskii~\cite{pol17}, see also \cite{nps17, ns17_2}. Here we included the argument for the bounded case for completeness.

\subsection{Proof of Theorem~\ref{th_lp_plane_strong}}

When $\M$ is unbounded, the result follows directly from \Cref{unboundedM}. Assume for the rest of the proof that $\M$ is bounded. By \Cref{distDist}, we can also assume that all pairwise distances are different in $\M$. Recall that the $\ell_p$-norm is strictly convex, and thus for all three (non-)collinear points in $\M$, their images are also (non-)collinear in every isometric copy of $\M$.

First, suppose that there exists an infinite collinear subset $\M_1 = \{\x^1,\x^2,\dots\} \subset \M$. 
It is clear that every polychromatic $\N$-colouring of $\mathcal{H}=\mathcal{H}(\R^2, \ell_p, \M_1)$ is also a polychromatic $\N$-colouring of $\mathcal{H}(\R^2, \ell_p, \M)$, so let us consider only the former hypergraph. 

For any hyperedge $e$ of $\mathcal{H}$, there exists only one $\ell_p$-isometric bijection between $e$ and $\M_1$, i.e. for all $i\in \N$, we can uniquely reconstruct the point $e(i) \in e$ playing the role of $\x^i$, because the distances between the points of $\M_1$ are pairwise distinct. 
Moreover, it is clear that a point on a line is uniquely determined by its distances to two other fixed points on this line. Summarising the above, we conclude that one can reconstruct $e$ knowing the positions of only three 
of its points. In other words, the hypergraph $\mathcal{H}(\R^2, \ell_p, \M_1')$ satisfies the conditions of \Cref{ErdosHajnal} with $k=3$, which concludes the proof in this case.

Now consider the remaining case, when $\M$ does not contain infinitely many collinear points. We claim that in this case, there exists an infinite subset $\M_2 = \{\z^1,\z^2,\dots\} \subset \M$ such that no three of its points are collinear. Indeed, one can construct this set recursively, since on every step, previously taken points generate only finitely many lines, and each of these lines excludes only finitely many points of $\M$ from consideration.

As earlier, it is sufficient to prove the existence of a polychromatic $\N$-colouring only for 
$\mathcal{H}=\mathcal{H}(\R^2, \ell_p, \M_2)$. Unfortunately, in this case, we might not be able to uniquely determine a point in the plane knowing only its distances to some $k$ fixed points, regardless of the value of $k$, even though $\M_2$ is `in general position'. Indeed, if all these $k$ point lie on the $\ell_p$-\textit{bisector} of $\y^1$ and $\y^2$ defined by

\begin{equation*}
	B_{\ell_p}(\y^1,\y^2) \coloneqq \{\x \in \R^2: \|\x-\y^1\|_p = \|\x-\y^2\|_p\},
\end{equation*}
then we cannot distinguish $\y^1$ from $\y^2$. Hence, the hypergraph $\mathcal{H}$
might fail to satisfy the conditions of \Cref{ErdosHajnal} for every fixed $k \in \N$. However, this situation is rather degenerate, and does not occur very often, as the following result suggests.

\begin{Theorem}[{Garibaldi, \cite[Section~5.4]{Garibaldi}}] \label{garibaldi}
	Let $1 < p < \infty$. Given two distinct pairs $(\y^1, \y^2)$ and $(\y^3, \y^4)$ of different points in the plane, if both bisectors $B_{\ell_p}(\y^1,\y^2)$ and $B_{\ell_p}(\y^3,\y^4)$ are non-linear, then they have  
 less than $k_0 \coloneqq 36\cdot2^{36}$ common points\footnote{It was conjectured in~\cite{Garibaldi}, that the actual maximum number of common points of distinct non-linear $\ell_p$-bisectors should be much smaller, `perhaps as low as $5$ for all $p$'.}.
\end{Theorem}

We show that this result implies that the hypergraph $\mathcal{H}$ satisfies the conditions of \Cref{infcommon} with $k=k_0 = 36\cdot2^{36}$. That is, we prove that for every $\ell_p$-isometric copy $e \subset \R^n$ of $\M_2$, its intersection $e'$ with all the other $\ell_p$-isometric copies of $\M_2$ that share at least $k$ points with $e$ is infinite. In fact, we will show that $|e\sm e'| \leq 1$ which is an even stronger property.
	
First, we observe that $|e\sm f| \leq 1$ for any $\ell_p$-isometric copy $f$ of $\M_2$ such that $|e\cap f|\ge k$. Indeed, since all the pairwise distances between the points of $\M_2$ are distinct by our assumption, it is easy to see that the common points of $e$ and $f$ correspond to the same indices in $\M_2$, i.e. $e(i_1)=f(i_1),\dots,e(i_k)=f(i_k)$ for some $k$ integers $i_1,\dots,i_k$. If $e(j)\neq f(j)$ for some $j \in \N$, then these $k$ points clearly lie on the bisector $B_{\ell_p}\big(e(j),f(j)\big)$ because
\begin{equation*}
	\|e(i)-e(j)\|_p = \|\z^i-\z^j\|_p = \|f(i)-f(j)\|_p
\end{equation*}
for all $i\in \N$. Note that these bisectors are distinct for different values of $j$. Besides, they are also non-linear, since $k \ge 3$, while no three points of $\M_2$ are collinear. Now \Cref{garibaldi} yields that there are less than two values of $j \in \N$ such that $e(j)\neq f(j)$. Hence, we indeed have $|e\sm f| \leq 1$, as claimed.
	
Now we prove a stronger inequality $|e\sm e'| \leq 1$. Assume the contrary, namely that for some integers $j_1, \, j_{\pot}$, we have $e(j_q) \notin e'$ for $q =1,2$. By the definition of $e'$, there are two not necessarily distinct $\ell_p$-isometric copies $f_1, \, f_{\pot}$ of $\M_2$ such that $|e\cap f_q|\ge k$ and $e(j_q) \notin f_q$ for $q = 1, \, 2$. Recall from the previous paragraph that $|e\sm f_q| \leq 1$ for each $q$. This implies that $\big|e\sm (f_1 \cap f_2)\big|\le 2$. In particular, we can choose $k$ points $e(i_1), \, \dots, \, e(i_k)$ that belong to the intersection $f_1 \cap f_2$. Let us also recall that the common points of $e$ and $f_q$ have the same indices, i.e. $e(i_p) = f_q(i_p)$ for all $p \in [k], q = 1,2$. As earlier, it is not hard to see that each of the $k$ points $e(i_p)$, $p\in [k]$, belongs to two distinct bisectors $B_{\ell_p}\left(e(j_q),f_q(j_q)\right), q = 1, \, 2$. This contradiction with \Cref{garibaldi} finally yields that $|e\sm e'| \leq 1$, and completes the proof of \Cref{th_lp_plane_strong}.

\section{Polygonal norms -- Proof of Theorem~\ref{th_polygon}} \label{sec3}

As we mentioned in the introduction, \Cref{th_polygon} was obtained in~\cite{FKS} for the $\ell_\infty$-norm, i.e. in case the unit disc is a square. In the present section, we generalise essentially the same argument for all polygonal norms. 

\subsection{Notation and preliminaries}

Under the conditions of \Cref{th_polygon}, let the unit disc $P\coloneqq B_N(\mathbf{0}, 1)$ be a centrally symmetric convex $2m$-gon. Given $k \in [m]$, pick one side of $P$ from the $k$-th pair of its opposite sides, and let $\W^k$ be the vector connecting its vertices. Let $\V^k \in \R^2$ be a vector orthogonal to $\W^k$ such that the two lines containing the $k$-th pair of opposite sides of $P$ are defined by the equations $\langle \x, \V^k \rangle =1$ and $\langle \x, \V^k \rangle =-1$, where $\langle\cdot,\cdot\rangle$ stands for the standard Euclidean dot product in the plane. Note that here and in what follows we do not distinguish points from their position vectors. Both vectors $\W^k$ and $\V^k$ are well-defined up to a sign, which we pick arbitrarily. Using this notation, it is not hard to see that
\begin{equation*}
	P = \{\x\in\R^2 : |\langle \x, \V^k \rangle|\le 1 \mbox{ for all } 1\le k \le m\},
\end{equation*}
and, even more generally, we have
\begin{equation*}
	\|\x\|_N = \max_{1\le k \le m} |\langle \x, \V^k \rangle|.
\end{equation*}

For a non-zero vector $\x \in \R^2$, let $k(\x)$ be an index such that $\|\x\|_N = |\langle \x, \V^{k(\x)} \rangle|$. In other words, the line $\{t\x: t \in \R\}$ intersects the polygon $P$ at the $k(\x)$-pairs of its opposite sides. If this line passes through two vertices  of the polygon, then $k(\x)$ is a $2$-valued function in this case.

\begin{Lemma} \label{lem_sum}
	Let $k \in [m]$ and $\x^1, \x^2 \in \R^2$ be vectors such that $\|\x^i\|_N = \langle \x^i, \V^k \rangle$, $i=1, \, 2$. Then we have $\|\x^1+\x^2\|_N = \langle \x^1+\x^2, \V^k \rangle$.
\end{Lemma}
\begin{proof}
	Indeed, for every $k' \in [m]$, we have $|\langle \x^1+\x^2, \V^{k'} \rangle| = |\langle \x^1, \V^{k'} \rangle + \langle \x^2, \V^{k'} \rangle| \le |\langle \x^1, \V^{k'} \rangle| + |\langle \x^2, \V^{k'} \rangle|$, while for $k'=k$ the inequality holds with equality, and the signs of these dot products are positive. Taking now the maximum over all $k' \in [m]$, we complete the proof. 
\end{proof}

Though the set $\G(q)\sm\{0\}$ is one-dimensional, not all of its $N$-isometric copies in the plane are collinear. However, the next lemma shows that each of them still has its `direction' in a sense that the $N$-distance between each two of its points is given by the dot product with a single vector $\V^{k}$. For an illustration, see \Cref{collinear}. The subsequent lemma then shows that there exists a whole segment of points extending a given $N$-isometric copy of $\G(q)\sm\{0\}$ to an $N$-isometric copy of $\G(q)$.

\begin{figure}[htp]
\centering
\includegraphics[width=9cm]{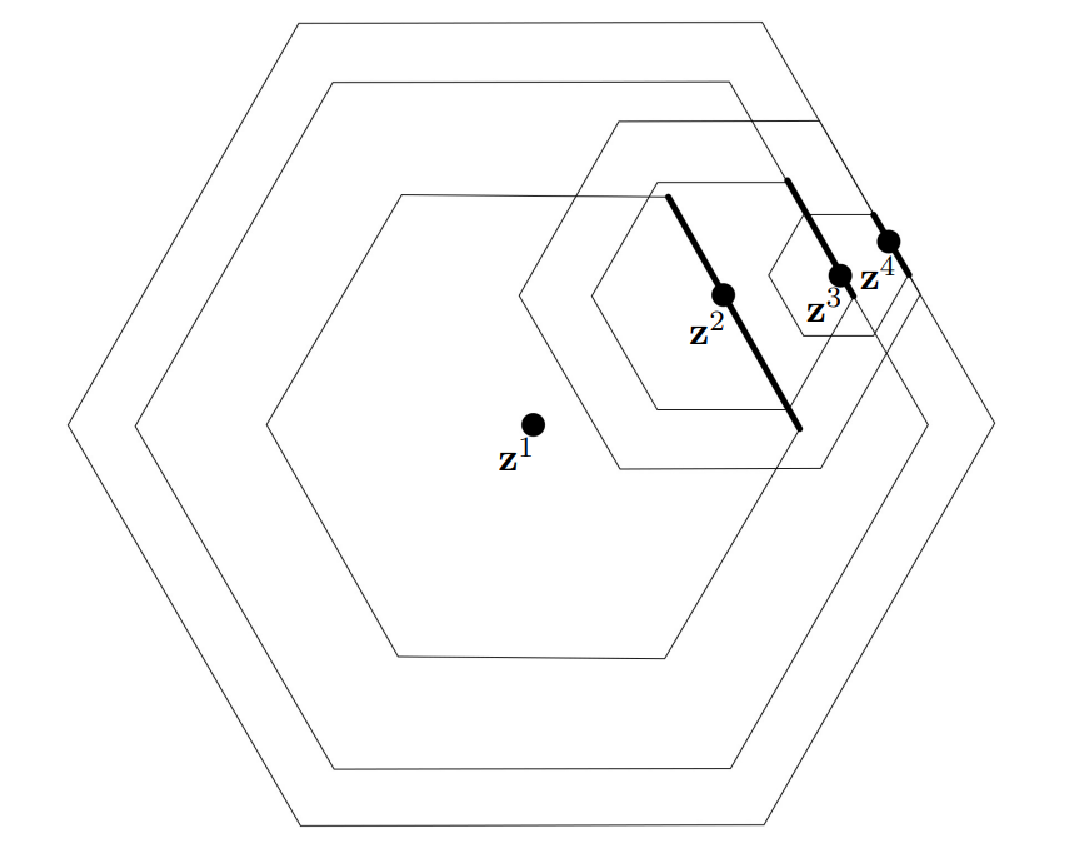}
\caption{Example of an embedding of the first four points of $\G(q)\sm\{0\}$ to the Minkowski plane equipped with a hexagonal norm. Even though the points are not collinear, they still have a direction defined by a single vector $\V^{k}$.}
\label{collinear}
\end{figure}

\begin{Lemma} \label{lem_direction}
	Let $\z^1, \z^2, \dots$ be a sequence of points in the plane such that $\|\z^i-\z^j\|_N = (q^i-q^j)/(1-q)$ for all $i<j$. Then there exists $k \in [m]$ and a sign $\sigma = \pm1$ such that $\|\z^i-\z^j\|_N = \sigma\langle \z^i-\z^j, \V^k \rangle$ for all $i<j$.
\end{Lemma}

\begin{proof}

	Let $\y \in \R^2$ be the limit of our sequence (which exists because $\z^1, \z^2, \dots$ is a Cauchy sequence). For every $i \in \N$, we have
	\begin{equation*}
		\|\z^i-\y\|_N = \lim_{j \to \infty} \|\z^i-\z^j\|_N = \frac{q^i}{1-q}.
	\end{equation*}
	Take an arbitrary $k \in [m]$ such that $|\langle \z^1-\y, \V^k \rangle| =  \|\z^1-\y\|_N = q/(1-q)$. Without loss of generality, assume that $\langle \z^1-\y, \V^k \rangle$ is positive. For every $i < j$, consider the following chain:
	\begin{align*}
		\frac{q}{1-q} =&\ \langle \z^1-\y, \V^k \rangle = \langle \z^1-\z^i, \V^k \rangle + \langle \z^i-\z^j, \V^k \rangle + \langle \z^j-\y, \V^k \rangle \\
		\le&\ \|\z^1-\z^i\|_N + \|\z^i-\z^j\|_N + \|\z^j-\y\|_N \\
		=&\ \frac{q-q^i}{1-q} + \frac{q^i-q^j}{1-q} + \frac{q^j}{1-q} = \frac{q}{1-q}.
	\end{align*}
	Thus, the latter inequality must hold with equality. In particular, this implies that
	\begin{equation*}
		\langle \z^i-\z^j, \V^k \rangle = \frac{q^i-q^j}{1-q} = \|\z^i-\z^j\|_N, 
	\end{equation*}
	as desired.
\end{proof}

Finally, we prove the following lemma which we will later use to extend a given monochromatic sequence of points to a copy of our infinite set.

\begin{Lemma} \label{lem_extend}
Let $\z^1, \z^2, \dots$ be a sequence of points in the plane such that  $\|\z^i-\z^j\|_N = (q^i-q^j)/(1-q)$ for all $i<j$, and $k \in [m]$ be from the statement of \Cref{lem_direction}. Let $I$ be the side of $P$ from the $k$-th pair of its opposite sides that contains a unit vector $\U \coloneqq (\z^1-\z^2)/q$. Then adding an arbitrary point $\z^0$ from the translated segment $\z^1+I$ to the sequence $\z^1, \z^2, \dots$ yields an $N$-isometric copy of $\G(q)$.
\end{Lemma}

\begin{proof}
	We can assume without loss of generality that $\sigma$ from the statement of \Cref{lem_direction} equals $1$, i.e. all the scalar products $\langle \z^i-\z^j, \V^k \rangle$, $i<j$, are positive. Then note that $\|\U\|_N = \|\z^1-\z^2\|_N/q = 1$, and similarly $\langle \U, \V^k \rangle = \langle \z^1-\z^2, \V^k \rangle/q = 1$. Therefore, $\U$ indeed belongs to a side $I$ of $P$ from the $k$-th pair of its opposite sides. Given any other point $\U' \in I$ on this side, it is clear that we have $\|\U'\|_N = \langle \U', \V^k \rangle = 1$ as well. Consider a point $\z^0 = \z^1+\U'$. On the one hand, we have $\|\z^0-\z^1\|_N = \langle \z^0-\z^1, \V^k \rangle = 1$ by construction. On the other hand, given $i>1$, \Cref{lem_direction} implies that $\|\z^1-\z^i\|_N = \langle \z^1-\z^i, \V^k \rangle = (q-q^i)/(1-q)$. Hence, it follows from \Cref{lem_sum} that
	\begin{equation*}
		\|\z^0-\z^i\|_N = \|(\z^0-\z^1)+(\z^1-\z^i)\|_N = \langle (\z^0-\z^1)+(\z^1-\z^i), \V^k \rangle = 1+\frac{q-q^i}{1-q} = \frac{1-q^i}{1-q}.
	\end{equation*}
	In particular, this shows that the sequence $\z^0, \z^1, \dots$ is indeed an $N$-isometric copy of $\G(q)$.
\end{proof}

\subsection{Proof of Theorem~\ref{th_polygon}}
As before, let
$	\G(q) \coloneqq \{0, \, 1, \, 1+q, \, 1+q+q^2, \, \dots\} $. We need to show that for a given polygonal norm $N$ in the plane we have $\chi\big(\R^2, N, \G(q)\big) \ge 3$ whenever $q$ is sufficiently small.

Assume the contrary and fix a polygonal norm $N$, a positive $q < \lambda/(1+\lambda)$ where $\lambda$ is the smallest $N$-distance between two
consecutive vertices of its unit disc, and a 
red-blue colouring of the plane such that no $N$-isometric copy of $\G(q)$ is monochromatic.

In the first step of the proof, we show that there exists a monochromatic segment in the plane. Indeed, pick an arbitrary vector $\x \in \R^2$ of unit norm such that the line $\{t\x: t \in \R\}$ does not pass through a vertex of $P$. Then the dot product $|\langle \x, \V^{k(\x)} \rangle|$ is strictly larger than $|\langle \x, \V^{k} \rangle|$ for every $k \neq k(\x)$. Thus there exists some sufficiently small $\tau>0$ such that $\|\x+t\W^{k(\x)}\|_N = \|\x\|_N = 1$ for all $t \in [-\tau; \tau]$. For every $i \ge 0$, we define a segment $J_i$ by
\begin{equation*}
	J_i = \big\{(1+q+\dots+q^{i-1})\x + t\W^{k(\x)}: |t| \le \tau q^{i}/2\big\}.
\end{equation*}
Assume that no segment is monochromatic. Then, all segments $J_i$, $i \ge 0$ must contain both red and blue points, so we can pick a red point $\z^i$ from each $J_i$. We claim that the sequence $\z^0, \z^1, \dots$ is an $N$-isometric copy of $\G(q)$, which would yield a contradiction. To see this, it is sufficient to check that for all $0 \le i < j$, we have $\|\z^j-\z^i\|_N = q^{i}+\dots+q^{j-1}$. Observe that
\begin{equation*}
	\z^j-\z^i = (q^{i}+\dots+q^{j-1})\x + t_{ij}\W^{k(\x)}, \mbox{ where } |t_{ij}| \le \tau(q^{i}+q^{j})/2 < \tau q^i \le \tau (q^{i}+\dots+q^{j-1}).
\end{equation*}
Hence, we indeed have
\begin{equation*}
	\|\z^j-\z^i\|_N = \|(q^{i}+\dots+q^{j-1})\x + t_{ij}\W^{k(\x)}\|_N = \|(q^{i}+\dots+q^{j-1})\x\|_N = q^{i}+\dots+q^{j-1},
\end{equation*}
as desired. So, for the rest of the proof we assume that there is a monochromatic segment.

Let $s \ge 1$ be the minimum number such that there exists an $N$-isometric monochromatic 
copy of $\G(q)$ minus its first $s$ points.  
In other words, we assume that some $N$-isometric copy of a `scaled' progression  $q^s\G(q)$ is monochromatic while none of $q^{s-1}\G(q)$ is. It is easy to see that the existence of a monochromatic segment we observed in the previous paragraph implies that $s$ is a well-defined integer. In the rest of this section, we complete the proof of Theorem~\ref{th_polygon} by obtaining a contradiction with the minimality of $s$. For the sake of simplicity, we provide our argument only for $s=1$, while it can be easily generalised. 
(One can actually assume that $s=1$ without loss of generality, since scaling of our colouring 
by the factor of $q^{s-1}$ produces a monochromatic $N$-isometric copy of $q\G(q)$, 
but not that of $\G(q)$.)

So, let $\z^1, \, \z^2, \, \dots $ be a red $N$-isometric copy of  $\G(q)\sm\{0\}$ (in this particular order). Since it cannot be completed by a point $\z^0$ to a red $N$-isometric copy of  $\G(q)$, \Cref{lem_extend} implies that there exists $k_0 \in [m]$ and an entirely blue segment $I_0$ parallel to $\W^{k_0}$. The \textit{length} of this segment, i.e.\ the $N$-distance between its endpoints, is equal to $\|\W^{k_0}\|_N = \lambda_{k_0} \ge \lambda > q/(1-q)$. Hence, we can inscribe into it a blue $N$-isometric copy of $\G(q)\sm\{0\}$. For definiteness, we inscribe it in such a way that the limit point of the copy coincides with the lexicographically greater endpoint of the segment. Applying \Cref{lem_extend} again, we see that some particular segment $I_1$ parallel to $\W^{k_1}$, where $k_1 = k(\W^{k_0})$, must be entirely red. Proceed similarly for every $j > 1$, finally we obtain a monochromatic segment $I_j$ parallel to 
$\W^{k_j}$ and of length $\lambda_{k_j}$, where $k_j = k(\W^{k_{j-1}})$. For an illustration, see \Cref{fig_thm2}.

\begin{figure}[htp]%
    \centering
    \subfloat[ Let $\z^1, \z^2, \dots \subseteq I_0$ be a red $N$-isometric copy of $\G(q)\sm\{0\}$ that cannot be extended to a red $N$-isometric copy of $\G(q)$. By \Cref{lem_extend}, there exists a blue segment $I_1$ of length at least $\lambda$.]{{\includegraphics[width=7cm]{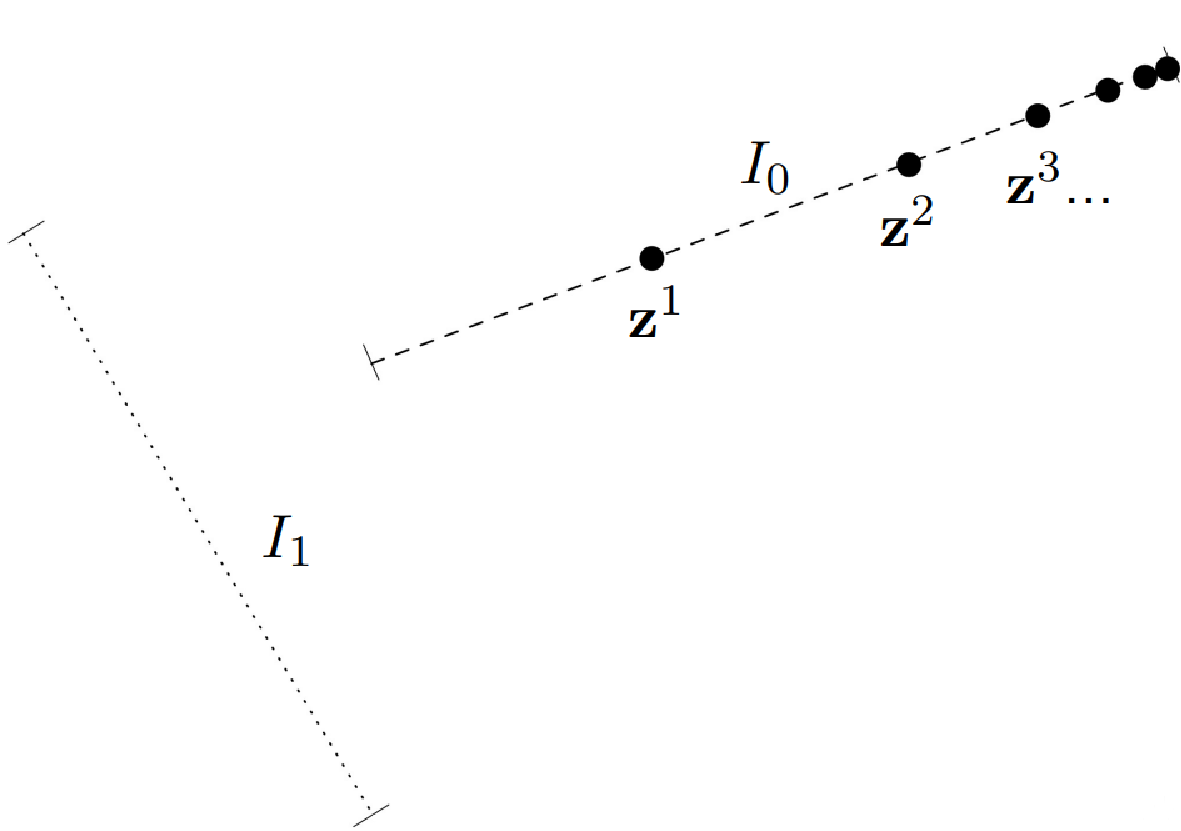} }}%
    \qquad
    \subfloat[As the length of $I_1$ is at least $\lambda$, we can inscribe into it a blue $N$-isometric copy of $\G(q)\sm\{0\}$. If it cannot be extended to a blue $N$-isometric copy of $\G(q)$, then again, by \Cref{lem_extend}, there exists a red segment $I_2$.]{{\includegraphics[width=8.9cm]{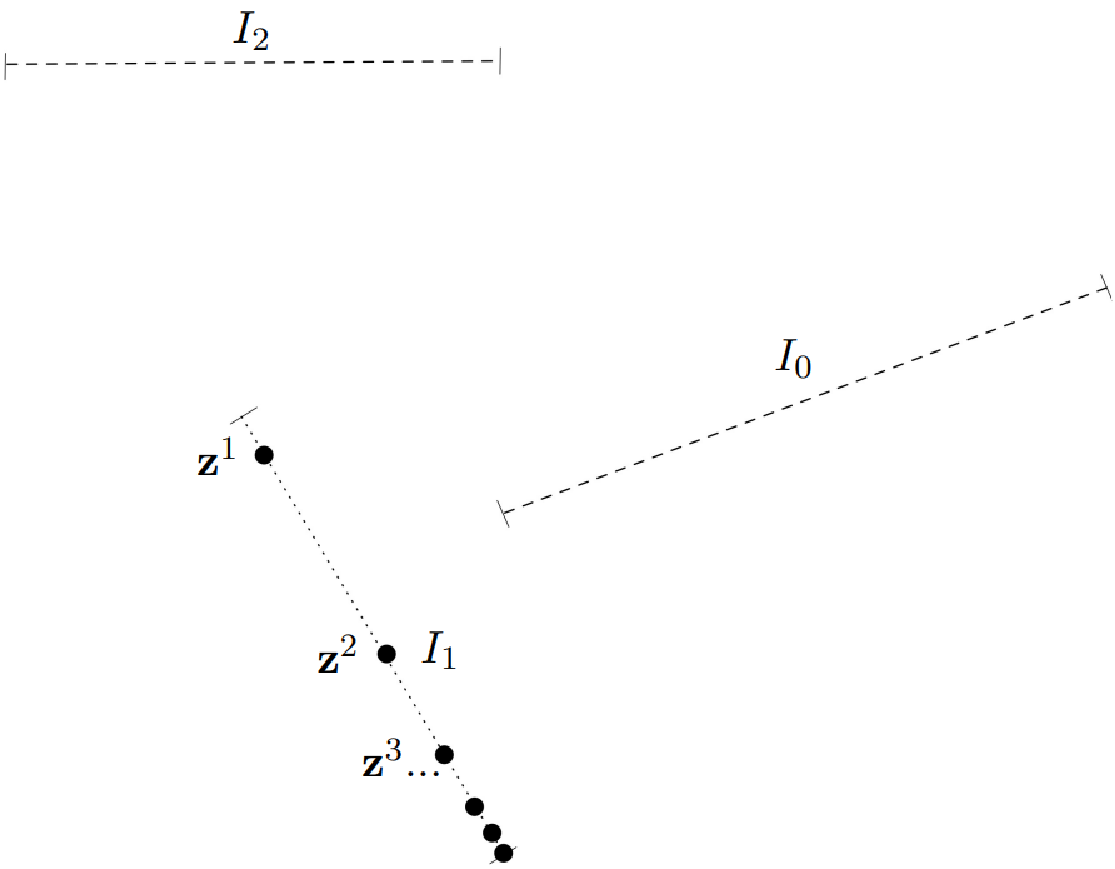} }}%
    \caption{Illustration of finding monochromatic segments $I_0$, $I_1, \dots$ in the Minkowski plane.}%
    \label{fig_thm2}%
\end{figure}

Since there are only finitely many possible directions, namely exactly $m$, we have $k \coloneqq k_{j_1} = k_{j_2}$ for some $j_1<j_2$. In other words, there exists a monochromatic segment $I_1' \coloneqq I_{j_1}$ parallel to $\W^k$ and of length $\lambda_{k}$. Moreover, if we inscribe into it an $N$-isometric copy $\M'$ of $\G(q)\sm \{0\}$ in  such a way that the limit point of $\M'$ coincides with the lexicographically greater endpoint of $I_1'$, and apply \Cref{lem_extend} several times, namely exactly $(j_2-j_1)$ times,
then we obtain another monochromatic segment $J_1'$ parallel to $I_1'$ and of the same length. Observe that the position of $J_1'$ linearly depends on the position of $\M'$ as a composition of several applications of \Cref{lem_extend} each of which is linear. In other words, a translation of $\M'$ results in the same translation of $J_1'$. Hence, as $\M'$ slides from one endpoint of $I_1'$ to another, the segment $J_1'$ slides along itself, resulting in a monochromatic segment $I_2'$ which is still parallel to $\W^k$ but of length $\lambda_{k}+ (\lambda_{k}-q/(1-q)) > \lambda_{k}$. Proceeding in the same vein, we find a family of monochromatic segments $I_1', \, I_2', \, \dots$ whose lengths form an increasing arithmetic progression with the common difference of $\lambda_{k}-q/(1-q)$. After finitely many steps, we get a monochromatic segment $I_j'$ of length larger than $1/(1-q)$. Inscribing into it an $N$-isometric copy of $\G(q)$ yields the desired contradiction.

\section{Concluding remarks, open problems}\label{sec:conc}

\noindent \textbf{Strictly convex norms.}
It is tempting to conjecture that for $n\ge 2$ and any infinite $\M \subset \R^n$, the equality $\chi(\R^n, N, \M)=2$ holds for a relatively wide class of norms, in particular, for strictly convex $N$. However, the only upper bound we have in the general case is based on the trivial inequality $\chi(\R^n, N, \M) \le \chi(\R^n, N)$.
In the planar case, Chilakamarri~\cite{Chil} showed that $\chi(\R^2, N) \le 7$ for every norm $N$. In higher dimension, Kupavskii~\cite{Kup}, improving on the result of F\"uredi and Kang~\cite{FK} proved that $\chi(\R^n, N) \le \big(4+o(1)\big)^n$, where the $o(1)$-term depends only on $n$. 
For $\ell_p$-norms, there is a better exponential upper bound \cite{Kup}.

We can adapt our argument from \Cref{sec2} for norms in higher dimension that satisfy a certain analogue of the conditions of \Cref{garibaldi} and obtain the following result.

\begin{Theorem} \label{garibaldi_norm}
Let $N$ be a strictly convex norm on $\R^n$. Suppose that there exist $m_0, \, k_0 \in \N$ such that the bisectors of any $m_0$ distinct pairs of points in $\mathbb{R}^n$ share at most $k_0$ points. Then $\chi(\R^n, N, \M)=2$ for all infinite $\M \subset \R^n$. 
\end{Theorem}

Observe that 
$\ell_p$-norms on $\R^n$
fail to satisfy the conditions of \Cref{garibaldi_norm} for $1<p<\infty$ and $n\ge 2$, 
since arbitrary many pairs of points can have the same line or hyperplane as their bisectors. 
If this situation is exceptional 
for a given norm, then we can try to describe all these `special bisectors' and treat them separately. 
For instance, \Cref{garibaldi} implies that the only special $\ell_p$-bisectors in the plane are lines. 
In addition, since $\ell_p$-norms are strictly convex for $1<p<\infty$, we know that if 
an infinite 
$\M\subset \R^2$ lies on a line, then every $\ell_p$-isometric copy of $\M$ lies on a line as well, and the argument is secured.

In the Euclidean case, the bisectors are hyperplanes. Moreover, if an infinite $\M\subset \R^n$ lies on a  hyperplane, then every $\ell_2$-isometric copy of $\M$ lies on a hyperplane as well, and the argument leads to the desired equality $\chi(\R^n, \ell_2, \M)=2$.

To extend these ideas at least to all $\ell_p$-norms on $\R^3$, one first needs a better understanding of the special bisectors, e.g.\ of planes.
Let $1<p<\infty$ and $\M \subset \R^3$ be an infinite subset of a horizontal plane. Does every $\ell_p$-isometric copy of $\M$ lie on a plane as well? 
Or at least, does 
every $\ell_p$-isometric copy of $\M$ contain an 
infinite subset that 
lie on a plane? 

\vspace{3mm}
\noindent \textbf{Polygonal norms.}
Another conjecture is that for $n \ge 3$ and every polygonal norm $N$ on $\R^n$, there exists an infinite $\M$, 
perhaps a geometric progression $\G(q)$ for a sufficiently small $q$, 
such that $\chi(\R^n, N, \M) \ge n+1$. As we discussed in the introduction, 
the latter holds for the $\ell_\infty$-norm, see \cite[Section~2.3]{FKS}. 
In this norm the unit ball is a hypercube, all of its faces are hypercubes of one less dimensions, 
and the induced norms on the faces are also the $\ell_\infty$-norms. 
These properties allowed the authors of \cite{FKS} to use induction on the dimension. 
We believe that it should be possible to combine the ideas from \cite{FKS} 
with those from our \Cref{sec3} to verify the above conjecture. 
However, in this general case, the induction step for the polygonal norm $N$ on $\R^n$ would 
utilise several polygonal norms on $\R^{n-1}$ 
induced on the faces of the unit $N$-ball.

For some special sequences of polygonal norms $N_n$ on $\R^n$, $n \in \N$, 
e.g.\ for the $\ell_1$-norm which exists in every dimension, it is natural to ask if there exists a single infinite $\M$ such that $\chi(\R^n, N_n, \M)$ tends to infinity with $n$. The only known result of this sort so far was obtained in \cite[Section~2.3]{FKS} and covers the case of the 
$\ell_\infty$-norms.

We hesitantly suspect that $\chi(\R^n, N, \M)$ cannot be larger than $n+1$ for any polygonal norm $N$ or $\R^n$ and any infinite $\M \subset \R^n$, since this upper bound was obtained in \cite[Section~2.3]{FKS} for the $\ell_\infty$-norm. A direct application of their argument to this general problem leads to the upper bound $\chi(\R^n, N, \M) \le m+1$, where $2m$ is the number of facets of the $N$-ball. 
Note that $m$ can be arbitrary large even if the dimension $n$ is fixed. 
The best upper bound we have as a function of $n$ is the trivial $\chi(\R^n, N, \M) \le \chi(\R^n, N)$. Most likely it is very far from the truth.  
Note that the general upper bound $\chi(\R^2, N) \le 7$ has been improved to  
$\chi(\R^2, N) \le 6$ for many special polygonal norms including those whose unit ball is a regular polygon with even and at most $22$ sides, see~\cite{Geher}.

\vspace{3mm}
\noindent \textbf{Mixed norms.}
We claim that our argument from \Cref{sec3} works essentially without any changes for the following class of norms in the plane (see \Cref{fig_mixnorm}).

\begin{Theorem} \label{mixed_norm}
    Let $N$ be a norm in the plane. Suppose that there exist two non-zero vectors $\x^1, \, \x^{2}$ 
    with the following properties:
    (i) the boundary of the unit $N$-ball contains line segments parallel to $\x^1$ and $\x^2$; 
    (ii) a line passing through the origin that is parallel to $\x^1$ (resp. $\x^2$) intersects the line segments parallel to $\x^{2}$ (resp. $\x^1$) on the boundary of the unit $N$-ball. Then $\chi\big(\R^2, N, \G(q)\big) \ge 3$ for all sufficiently small positive $q$.
\end{Theorem}

\begin{figure}[htp]
\centering
\includegraphics[width=10cm]{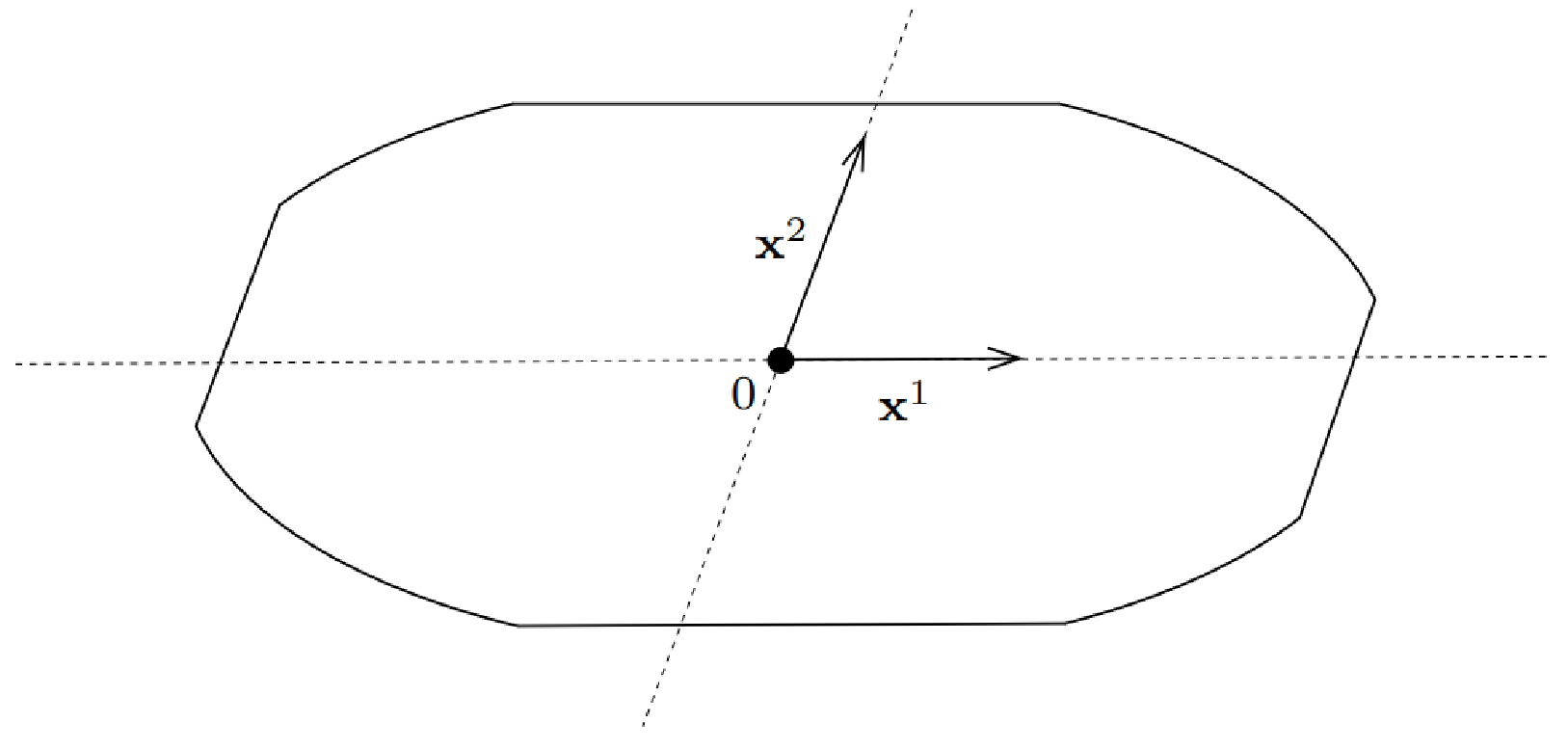}
\caption{An example of a mixed
norm that satisfies the conditions of \Cref{mixed_norm}.}
\label{fig_mixnorm}
\end{figure}

We are not brave enough to conjecture what happens if a `mixed-norm' fails to satisfy the conditions of \Cref{mixed_norm}.

\vspace{5mm}

\noindent
{\bf \large Acknowledgements.} We thank the anonymous referees for their valuable suggestions. We also thank Peter Allan for his suggestions to improve the wording of the introduction. The authors were supported by ERC Advanced Grant `GeoScape' No.\ 882971. Panna Geh\'er was also supported by the Lend\"ulet Programme of the Hungarian Academy of Sciences -- grant number LP2021-1/2021. G\'eza T\'oth was also supported by the National Research, Development and Innovation Office, NKFIH, K-131529.

\end{document}